\documentclass[12pt,reqno]{amsart}

\usepackage{amssymb}
\usepackage{amsthm}
\usepackage{amsaddr}
\usepackage{enumitem}

\usepackage{geometry}
\newtheorem{theorem}{Theorem}
\newtheorem{corollary}{Corollary}
\newtheorem{example}{Example}

\newcommand{\R}{\mbox{$\mathbb{R}$}}
\newcommand{\C}{\mbox{$\mathbb{C}$}}
\newcommand{\F}{\mbox{$\mathbb{F}$}}

\begin{document}

\title[Continuity and Discontinuity of Seminorms]
{Continuity and Discontinuity of Seminorms\\ on Infinite-Dimensional Vector Spaces. II}

\author[Jacek Chmieli\'nski]{Jacek Chmieli\'nski}
\address{Department of Mathematics, Pedagogical University of Krak\'ow\\ Krak\'{o}w, Poland\\
E-mail: jacek.chmielinski@up.krakow.pl}

\author[Moshe Goldberg]{Moshe Goldberg$^{1}$}\thanks{$^{1}$Corresponding author}
\address{Department of Mathematics, Technion -- Israel Institute of Technology\\ Haifa, Israel\\
E-mail: mg@technion.ac.il}

\begin{abstract}
In this paper we extend our findings in \cite{ChG} and answer further questions regarding continuity and discontinuity of seminorms on infinite-dimensional vector spaces.
\end{abstract}


\maketitle

Throughout this paper let $X$ be a vector space over a field $\F$, either $\R$ or $\C$. As usual, a~real-valued function $N\colon X\to\R$ is a {\em norm} on $X$ if for all $x,y\in X$ and $\alpha\in\F$,
$$
\begin{array}{l}
N(x)>0,\quad x\neq 0,\\
N(\alpha x)=|\alpha| N(x),\\
N(x+y)\leq N(x)+N(y).
\end{array}
$$

Furthermore, a real-valued function $S\colon X\to\R$ is called a {\em seminorm} if for all $x,y\in X$ and $\alpha\in\F$,
$$
\begin{array}{l}
S(x)\geq 0,\\
S(\alpha x)=|\alpha| S(x),\\
S(x+y)\leq S(x)+S(y);
\end{array}
$$
hence, a norm is a positive-definite seminorm.

Using standard terminology, we say that a seminorm $S$ is {\it proper} if $S$ does not vanish identically and
$S(x)=0$ for some $x\neq 0$ or, in other words, if
$$
\ker S := \{x\in X:\ S(x) = 0\},
$$
is a nontrivial proper subspace of $X$.

Lastly, just as for norms, we say that seminorms $S_1$ and $S_2$ are {\it equivalent} on $X$, if there exist positive
constants $\beta\leq \gamma$ such that for all $x\in X$,
$$
\beta S_1(x) \leq S_2(x) \leq \gamma S_1(x).
$$

We recall that if $X$ is finite-dimensional, then all norms on $X$ are equivalent, thus giving rise to a unique norm-topology on $X$. This well-known fact leads to the following result.

\begin{theorem}[\cite{Go}]\label{t1}
Let $S$ be a seminorm on a finite-dimensional vector space $X$ over $\F$. Then $S$ is continuous with respect to the unique norm-topology on $X$.
\end{theorem}

Unlike the finite-dimensional case, if $X$ is infinite-dimensional then not all norms on $X$ are equivalent, and accordingly we no longer have a unique norm-topology. Indeed, in our previous paper \cite{ChG} we explored continuity and discontinuity of seminorms when the assumption of finite-dimensionality was removed. Our two main findings were:

\begin{theorem}[{\cite[Theorem 2]{ChG}}]\label{t2}
Let $X$, an infinite-dimensional vector space over $\F$, be equipped with a seminorm $S$ and a norm $N$. Then:
\begin{enumerate}[label={\rm (\alph*)}]
\item\label{t2a}
$S$ is ubiquitously continuous in $X$ with respect to the topology induced by $N$ if and only if there exists some point of $X$ at which $S$ is continuous.
\item\label{t2b}
Similarly, $S$ is ubiquitously discontinuous in $X$ with respect to the above mentioned topology if and only
if there exists some point of $X$ at which $S$ is discontinuous.
\end{enumerate}
\end{theorem}

\begin{theorem}[{\cite[Theorem 3]{ChG}}]\label{t3}
Let $S\neq 0$ be a seminorm on an infinite-dimensional vector space $X$ over $\F$. Then:
\begin{enumerate}[label={\rm (\alph*)}]
\item\label{t3a}
There exists a norm with respect to which S is ubiquitously continuous in X.
\item\label{t3b}
There exists a norm with respect to which S is ubiquitously discontinuous in X.
\end{enumerate}
\end{theorem}

As it is, the above results trigger several new questions which are the main business of the present paper.

Having a second look at Theorem \ref{t1}, we begin our quest by asking whether there exists an infinite-dimensional normed space $X$ which resembles to some extent the behavior of the finite-dimensional case in the sense that {\it all} seminorms on that space are continuous with respect to the given norm topology in $X$. As
we shall see next, the answer to this question is negative:

\begin{theorem}\label{t4}
Let $N$ be a norm on an infinite-dimensional vector space $X$ over $\F$. Then there exist a norm $N'$ as well as a proper seminorm $S$ which are ubiquitously discontinuous in $X$ with respect to the norm topology induced by $N$.
\end{theorem}
\begin{proof}
Let $\mathcal{B}_N:=\{x\in X:\ N(x)\leq 1\}$ be the unit ball in our normed space. Since ${\rm span}\,\mathcal{B}_N = X$, we consult
Corollary 4.2.2 in \cite{K} which ensures us that $\mathcal{B}_N$
contains a Hamel basis $H$ for $X$. Now, fix a sequence
$\{h_1,h_2,h_3,\ldots\}$ of distinct elements in $H$ and set
$$
g_n=\frac{h_n}{n},\qquad n=1,2,3,\ldots.
$$
Then
$$
B :=\{g_1, g_2, g_3,\ldots\}\cup \left( H\smallsetminus\{h_1,h_2,h_3,\ldots \}\right)
$$
is a Hamel basis of $X$ as well. So every $x$ in $X$ assumes a unique representation of the form
$$
x=\sum_{b\in B}\alpha_b(x)b,\qquad \alpha_b(x)\in\F,
$$
where $\{b\in B:\ \alpha_b(x)\neq 0\}$ is a finite set.

With this representation at our disposal, we can easily confirm that the real-valued functions
$$
N'(x):=\sum_{b\in B}|\alpha_b(x)|,\qquad x\in X,
$$
and
$$
S(x):=\sum_{b\in B\smallsetminus\{g_1\}}|\alpha_b(x)|,\qquad x\in X,
$$
are a norm and a proper seminorm on $X$, respectively.

We get, however, that
$$
N(g_n)=\frac{1}{n}N(h_n)\leq\frac{1}{n},\qquad n=1,2,3,\ldots,
$$
whereas,
$$
N'(g_n)=1,\qquad n=1,2,3,\ldots,
$$
and
$$
S(g_n)=1,\qquad n=2,3,4,\ldots.
$$
Consequently, $N'$ and $S$ are discontinuous at zero with respect to $N$; so by Theorem \ref{t2}\ref{t2b}, both $N'$ and $S$
are ubiquitously discontinuous in $X$ with respect to the norm topology induced by $N$ and our assertion follows.
\end{proof}

In view of Theorem \ref{t4}, it seems natural to ask whether, given a norm $N$ on $X$, there exists another norm or a proper seminorm which are ubiquitously continuous in $X$ with respect to $N$. Surely, every norm which is equivalent to $N$ provides a positive answer to the above question. If, however, we look for a non-equivalent norm which is ubiquitously continuous with respect to the topology induced by $N$, we have nothing to say, unless $X$ is complete with respect to both norms---a case discussed in Corollary \ref{c1} later on.
Turning to proper seminorms, the answer to our question is positive if somewhat involved:

\begin{theorem}\label{t5}
Let $N$ be a norm on an infinite-dimensional vector space $X$ over $\F$. Then there exists a proper seminorm $S$ which is ubiquitously continuous in $X$ with respect to $N$.
\end{theorem}
\begin{proof}
Let $V$ be a nontrivial finite-dimensional subspace of $X$, and let $U$ be another subspace of $X$ such that $X=U\oplus V$. Furthermore, let $N'\colon U\to\R$ be the real-valued function that measures, with respect to $N$, the distance from elements in $U$ to the subspace $V$; that is,
\begin{equation}\label{e1}
N'(u):={\rm dist}_N(u,V)=\inf_{v\in V}N(u-v),\qquad u\in U.
\end{equation}
Since $V$ is finite-dimensional, the infimum in \eqref{e1} is attained, so we can write
$$
N'(u)=\min_{v\in V}N(u-v),\qquad u\in U.
$$

We shall now show that $N'$ is a norm on $U$. Indeed, if $N'(u) = 0$ for some $u\in U$, then the closedness of $V$ implies that $u\in V$; thus $u = 0$, and it follows that
$$
N'(u) > 0\quad \mbox{for all}\ u \neq 0.
$$
Moreover, for $u\in U$ and $\lambda\in F\smallsetminus\{0\}$ we get
$$
N'(\lambda u)=\min_{v\in V}N(\lambda u-v)= \min_{v\in V}N(\lambda u - \lambda v)=|\lambda|\min_{v\in V}N(u-v)=|\lambda|N'(u).
$$
And since $N'(0) = 0$, we infer that
$$
N'(\lambda u) = |\lambda| N'(u)\quad \mbox{for all}\ u \in U\ \mbox{and}\  \lambda\in F.
$$
Finally, select $u_1,u_2 \in U$. Then there exist elements $v_1,v_2\in V$ such that
$$
N'(u_1) = N(u_1-v_1)\quad \mbox{and}\quad N'(u_2) = N(u_2-v_2).
$$
Therefore,
\begin{eqnarray*}
N'(u_1+u_2)&=&\min_{v\in V}N(u_1+u_2-v)\leq N(u_1+u_2-(v_1+v_2))\\
&\leq& N(u_1-v_1) + N(u_2-v_2)=N'(u_1)+N'(u_2),
\end{eqnarray*}
and the fact that $N'$ is a norm on $U$ is in the bag.

Next, for any $x\in X$ and its unique decomposition $x = u+ v$ where $u\in U$ and $v\in V$, we put
$$
S(x):= N'(u).
$$
Since $N'$ is a norm on $U$, it is easily verified that the mapping $S\colon X\to\R$ is a seminorm on $X$.
Furthermore, since $\ker S = V$ and $V$ is a nontrivial proper subspace of $X$, we conclude that $S$ is a proper seminorm on X.

Lastly, for any $x = u+ v$ in $X$ with $u\in U$ and $v\in V$, we obtain
$$
S(x)=N'(u)=\min_{v'\in V}N(u-v')\leq N(u+v)=N(x).
$$
Hence, $S$ is majorized by $N$ on $X$, implying the continuity of $S$ at zero; so Theorem \ref{t2}\ref{t2a} forces the desired result.
\end{proof}

Falling back on Theorem \ref{t3}\ref{t3a}, we remember that every seminorm on an infinite-dimen\-sional vector space
$X$ is ubiquitously continuous with respect to some norm on $X$. {\it Assuming that a given seminorm is ubiquitously continuous with respect to two norms, we ask whether both norms are necessarily equivalent}. The following example answers this question in the negative:

\begin{example}\label{ex1}
{\rm
Let $c_{00}$ denote the familiar space of all infinite $\F$-valued sequences with finite support, i.e.,
$$
c_{00} = \{x = \{\xi_1,\xi_2,\xi_3,\ldots\}:\ \xi_i\in\F,\ \exists i_0\ \mbox{such that}\ \xi_i = 0\ \mbox{for all}\ i \geq i_0\}.
$$
Consider the proper seminorm
$$
S(x):= |\xi_1|,\qquad x=\{\xi_1,\xi_2,\xi_3,\ldots\}\in c_{00},
$$
and the two norms
$$
N_1(x):=\sum_{i=1}^{\infty}|\xi_i|,\qquad x=\{\xi_1,\xi_2,\xi_3,\ldots\}\in c_{00},
$$
$$
N_{\infty}(x):=\max_{i}|\xi_i|,\qquad x=\{\xi_1,\xi_2,\xi_3,\ldots\}\in c_{00}.
$$
Surely,
\begin{equation}\label{e2}
S(x) \leq N_{\infty}(x) \leq N_1(x)\quad \mbox{for all}\ x\in c_{00}.
\end{equation}

Now let $\{x_n\}_{n=1}^{\infty}$ be an arbitrary sequence in $c_{00}$. If either $N_1(x_n)\to 0$ or $N_{\infty}(x_n)\to 0$ as $n\to\infty$ then, by \eqref{e2}, $S(x_n)\to 0$, hence $S$ is continuous at zero with respect to $N_1$ and $N_{\infty}$; so by Theorem \ref{t2}\ref{t2a}, $S$ is ubiquitously continuous with respect to both norms.

As it is, however, $N_1$ and $N_\infty$ are non-equivalent. To justify this statement, set
$$
x_n=\{\underbrace{\frac{1}{n},\ldots,\frac{1}{n}}_{n\ \mbox{\tiny times}},0,0,0,\ldots\},\qquad n=1,2,3,\ldots.
$$
Then $N_1(x_n) =1$ for all $n$, whereas $N_{\infty}(x_n)=\frac{1}{n}\to 0$ as $n\to\infty$, and we are done.
}
\end{example}

We appeal now to the second part of Theorem \ref{t3} which tells us that every seminorm on $X$ is ubiquitously discontinuous with respect to some norm on $X$. In analogy to our previous question, {\it we assume that a given seminorm is ubiquitously discontinuous with respect to two norms on $X$, and ask whether these norms are
necessarily equivalent}. Again, the following example furnishes a negative answer:

\begin{example}\label{ex2}
{\rm
Let us resort to the space $c_{00}$ and to the non-equivalent norms $N_1$ and $N_{\infty}$ in Example \ref{ex1}. By Theorem \ref{t4}, there exists a proper seminorm $S$ which is ubiquitously discontinuous in $X$ with respect to $N_1$.
In particular, $S$ is discontinuous at zero, so for some sequence $\{x_n\}_{n=1}^{\infty}$ in $c_{00}$ we have $N_1(x_n)\to 0$ while
$S(x_n)\not\to 0$. Further, since $N_{\infty}\leq N_1$ on $c_{00}$, we get $N_{\infty}(x_n)\to 0$ which again, in view of $S(x_n)\not\to 0$, implies the discontinuity of $S$ at zero with respect to $N_{\infty}$. It thus follows that $S$ is ubiquitously discontinuous in $c_{00}$ with respect to $N_1$ and $N_{\infty}$, and our goal is achieved.
}
\end{example}

With Examples \ref{ex1} and \ref{ex2} in store, we assert that {\it the space $c_{00}$ is incomplete with respect to either of the
corresponding norms $N_1$ and $N_{\infty}$}.

To substantiate our claim, consider for example the sequence $\{x_n\}_{n=1}^{\infty}$ in $c_{00}$ where
$$
x_n=\{\frac{1}{2},\frac{1}{4},\ldots,\frac{1}{2^n},0,0,0,\ldots\},\qquad n=1,2,3,\ldots.
$$
For any positive integers $n > m$ we get
$$
N_1(x_n-x_m)=\sum_{i=m+1}^{n}\frac{1}{2^i}\leq\frac{1}{2^m}\quad \mbox{and}\quad N_{\infty}(x_n-x_m)=\frac{1}{2^{m+1}},
$$
which renders $\{x_n\}$ a Cauchy sequence with respect to both $N_1$ and $N_{\infty}$.

Now fix an arbitrary element $y = \{\eta_1,\eta_2,\ldots,\eta_k,0,0,0,\ldots \}$ in $c_{00}$. Then, for all $n$, $n > k$, we obtain
$$
N_1(x_n-y)\geq \frac{1}{2^{k+1}}\quad\mbox{and}\quad N_{\infty}(x_n-y)\geq \frac{1}{2^{k+1}}.
$$
Whence $y$ may not be the limit of $\{x_n\}$ with respect to either $N_1$ or $N_{\infty}$, and our assertion is validated.

In fact, {\it $c_{00}$ cannot be made into a Banach space, for it is incomplete with respect to any norm}. To verify
this statement we summon the canonical basis $\{e_n\}_{n=1}^{\infty}$ of $c_{00}$ where $e_n$ is the vector whose $n$-th entry is $1$ and all others vanish. Evidently, $\{e_n\}$ is a countable Hamel basis for $c_{00}$, while it is known (e.g., \cite{BB}, \cite{L}) that a Hamel basis of a (separable or not) infinite-dimensional Banach space is uncountable.

With the above observation in mind, we attend now to the case where $X$ is a complete space with respect to certain norms. Following standard nomenclature, we shall henceforth call a norm $N$ {\it complete} if $X$ is complete with respect to $N$.

\begin{theorem}\label{t6}
Let $N$ be a norm on an infinite-dimensional vector space $X$ over $\F$, such that $N$ is ubiquitously continuous with respect to two complete norms $N_1$ and $N_2$. Then $N_1$ and $N_2$ are equivalent.
\end{theorem}
\begin{proof}
Consider the norm
$$
N'(x) := \max\{N_1(x), N_2(x)\},\qquad x\in X ,
$$
and let us prove that $N'$ is complete. Select an arbitrary Cauchy sequence $\{x_n\}_{n=1}^{\infty}$ with respect to $N'$ so that
$$
N'(x_n - x_m)\to 0\quad \mbox{as}\ n,m\to\infty.
$$
Since $N'$ majorizes both $N_1$ and $N_2$, it follows that $\{x_n\}$ is a Cauchy sequence with respect to $N_1$ and $N_2$ as well.

Furthermore, since $N_1$ and $N_2$ are complete, we may exhibit elements $x',x''\in X$ such that
\begin{equation}\label{e3}
N_1(x_n - x')\to 0\quad \mbox{and}\quad N_2(x_n - x'')\to 0,\quad n\to\infty;
\end{equation}
so by the continuity of $N$ with respect to $N_1$ and $N_2$, we conclude that
$$
N(x_n - x')\to 0\quad \mbox{and}\quad N(x_n - x'')\to 0\quad \mbox{as}\ n\to\infty.
$$
Consequently,
$$
N(x' - x'') \leq N(x' - x_n ) + N(x_n - x'')\to 0\quad \mbox{as}\ n\to\infty;
$$
thus $N(x' - x'') = 0$, and we infer that $x' = x'' = \bar{x}$ for some $\bar{x}\in X$. By \eqref{e3} therefore, $N_1(x_n-\bar{x})\to 0$ and
$N_2(x_n-\bar{x})\to 0$; hence $N'(x_n-\bar{x})\to 0$ and the completeness of $N'$ is secured.

Finally, we invoke the well-known Banach inverse mapping theorem (e.g., \cite[Corollary 10.9]{Gi}), which
implies, \cite[Corollary 10.10]{Gi}, that {\it two complete norms on $X$ such that one majorizes the other, must be
equivalent}.

Indeed, since $N_1$, $N_2$ and $N'$ are complete, and as
$$
N_1(x) \leq N'(x),\quad N_2(x) \leq N'(x),\qquad x\in X,
$$
we deduce that $N_1$ and $N_2$ are each equivalent to $N'$. Whence $N_1$ and $N_2$ are equivalent, and the proof is at hand.
\end{proof}

Reflecting on Theorem \ref{t6}, we maintain that {\it the assumption of completeness of $N_1$ and $N_2$ cannot be dropped}. Indeed, revisit the space $c_{00}$ and the norms $N_1$ and $N_{\infty}$ in Example \ref{ex1}. Putting $N = N_{\infty}$ and employing the fact that $N_{\infty}\leq N_1$, we ascertain that $N$ is ubiquitously continuous with respect to both $N_1$ and $N_{\infty}$; yet, as shown in Example \ref{ex1}, $N_1$ and $N_{\infty}$ are non-equivalent.

\pagebreak

With Theorem \ref{t6} fresh in our mind, we may record the following simple proposition:

\begin{corollary}\label{c1}
Let $N_1$ and $N_2$ be two complete norms on an infinite-dimensional vector space $X$ over $\F$.
Then:
\begin{enumerate}[label={\rm (\alph*)}]
\item\label{c1a}
$N_1$ is ubiquitously continuous with respect to $N_2$ if and only if $N_1$ and $N_2$ are equivalent.
\item\label{c1b}
$N_1$ is ubiquitously discontinuous with respect to $N_2$ if and only if $N_1$ and $N_2$ are non-equivalent.
\item\label{c1c}
$N_1$ is ubiquitously continuous with respect to $N_2$ if and only if $N_2$ is ubiquitously continuous with respect to $N_1$.
\item\label{c1d}
$N_1$ is ubiquitously discontinuous with respect to $N_2$ if and only if $N_2$ is ubiquitously discontinuous with respect to $N_1$.
\end{enumerate}
\end{corollary}
\begin{proof}
To prove \ref{c1a}, set $N = N_1$. Then Theorem \ref{t6} forces the equivalence of $N_1$ and $N_2$ since $N_1$ and $N_2$ are complete, and $N$ is ubiquitously continuous with respect to both norms. Parts \ref{c1b}--\ref{c1d} are obtained from
part \ref{c1a} without much difficulty, and the corollary follows.
\end{proof}

Part \ref{c1a} of Corollary \ref{c1} tells us that a Banach space may not admit another non-equivalent complete norm which is ubiquitously continuous with respect to the original norm. {\it The question whether a Banach space may admit an incomplete norm which is non-equivalent as well as ubiquitously continuous with respect to the
original one}, is answered affirmatively by the following example.

\begin{example}\label{ex3}
{\rm
Consider the space $l_{\infty}$ of all infinite $\F$-valued bounded sequences    with the
usual norm
$$
N_{\infty}(x) := \sup_{i}|\xi_i|,\qquad x = \{\xi_1,\xi_2,\xi_3,\ldots\}\in l_{\infty}.
$$
It is well known (e.g., \cite[Proposition 1.16]{FHHMZ}) that equipped with the above norm, $l_{\infty}$ is a Banach space
over $\F$.

Define a second norm on $l_{\infty}$ by
$$
N'(x):=\sum_{i=1}^{\infty}2^{-i}|\xi_i|,\qquad x = \{\xi_1,\xi_2,\xi_3,\ldots\}\in l_{\infty}.
$$
Obviously,
$$
N'(x)\leq N_{\infty}(x),\qquad x\in l_{\infty};
$$
so $N'$ is ubiquitously continuous with respect to $N_{\infty}$.

Recall now the sequence $\{e_n\}_{n=1}^{\infty}$
where, as before, $e_n$ is the vector whose $n$-th entry is $1$ and all others vanish. Since
$$
N'(e_n)=2^{-n}\to 0\quad \mbox{as}\ n\to\infty,
$$
and
$$
N_{\infty}(e_n)=1\quad \mbox{for all}\ n=1,2,3,\ldots,
$$
the possibility that our two norms are equivalent shutters. Finally, we observe that $N'$ is incomplete on $l_{\infty}$
for otherwise, Theorem \ref{t6} would imply the equivalence of $N'$ and $N_{\infty}$, a contradiction.
}
\end{example}

Encouraged by the above example, we ask {\it whether every Banach space admits an incomplete norm which is non-equivalent as well as ubiquitously continuous with respect to the original norm}. The answer to this question remains open.

\bigskip

Bringing up Theorem \ref{t6} again, we shall show now that the role of the norm $N$ may not be replaced by a proper seminorm. More precisely, the following example will confirm that {\it if $S$ is a proper seminorm on an infinite-dimensional vector space $X$ over $\F$, such that $S$ is ubiquitously continuous with respect to two
complete norms $N_1$ and $N_2$, then $N_1$ and $N_2$ are not necessarily equivalent}.

\begin{example}\label{ex4}
{\rm
Let $l_1$ be the space of all absolutely summable $\F$-valued sequences with the familiar norm,
$$
N_1(x):=\sum_{i=1}^{\infty}|\xi_i|,\qquad x=\{\xi_1,\xi_2,\xi_3,\ldots\}\in l_1,
$$
and let $c_0$ be the space of all $\F$-valued sequences that converge to zero with
\begin{equation}\label{e4}
N_{\infty}(x):=\sup_{i}|\xi_i|, \qquad x=\{\xi_1,\xi_2,\xi_3,\ldots\}\in c_0.
\end{equation}

It is well known that furnished with the above norms, both $l_1$ and $c_0$ are separable Banach spaces (e.g., \cite[Propositions 1.16 and 1.42]{FHHMZ}). Hence (see \cite{L}), choosing Hamel bases, $B$ for $l_1$ and $B'$ for $c_0$, these bases must be of the same cardinality $\mathfrak{c}$; so there exists a bijection $f_0$ from $B$ onto $B'$. Now, as in the proof of Theorem \ref{t4}, each $x$ in $l_1$ takes on a unique representation of the form
$$
x=\sum_{b\in B}\alpha_b(x)b,\qquad \alpha_b(x)\in\F,
$$
where $\{b\in B:\ \alpha_b(x)\neq 0\}$ is a finite set. Thus
$$
f(x):=\sum_{b\in B}\alpha_b(x)f_0(b),\qquad x\in l_1,
$$
is a linear mapping from $l_1$ into $c_0$; and since $f_0$ is a bijection from $B$ onto $B'$, we leave it to the reader to
verify that $f$ is a linear bijection from $l_1$ onto $c_0$.

Next, for every $x = \{\xi_1,\xi_2,\xi_3,\ldots\}\in c_0$ we define three auxiliary mappings: The left shift
$$
L(x) := \{\xi_2,\xi_3,\xi_4,\ldots\},
$$
the right shift
$$
R(x) := \{0,\xi_1,\xi_2,\xi_3,\ldots\},
$$
and the truncation
$$
T(x) := \{\xi_1,0,0,0,\ldots \}.
$$
Obviously, $L$, $R$, and $T$ are linear mappings from $c_0$ into $c_0$. Further, since $l_1$ is a linear subspace of $c_0$,
we observe that $L$, $R$, and $T$ are linear from $l_1$ into $l_1$.

Aided by the above mappings, we define yet another map, $F\colon l_1\to c_0$, by
$$
F(x):=R(f(L(x)))+T(x),\qquad  x\in l_1.
$$
Clearly, $F$ is linear. Moreover, if $F(x) = 0$ then, by the definitions of $R$ and $T$, we note that both $R( f (L(x))) = 0$ and $T(x) = 0$. In addition, the first of these conditions implies that $f(L(x))=0$; and since $f$ is a bijection, it follows that $L(x) = 0$. This, together with \mbox{$T(x)=0$}, ensures that $x = 0$. Thus we have shown that $F(x) = 0$ implies $x = 0$, so $F$ is injective. Furthermore, selecting an arbitrary element $y\in c_0$, and setting \mbox{$x:=R(f^{-1}(L(y)))+T(y)$}, it is quite straightforward to confirm that $x$ belongs to $l_1$ and
$F(x)=y$; hence $F$ is surjective, and we conclude that $F$ is a {\it linear bijection} from $l_1$ onto $c_0$.

Introduce now the real-valued function $N_2\colon l_1\to\R$, defined by
$$
N_2(x):= N_{\infty}(F(x)),\qquad x\in l_1,
$$
where $N_{\infty}$ is the norm in \eqref{e4}. Since $F$ is a linear bijection and $N_{\infty}$ is a complete norm on $c_0$, it is
a routine matter, if somewhat tedious, to verify that $N_2$, just like $N_1$, is a complete norm on $l_1$, a task left to
the reader.

Next, consider a seminorm on $l_1$ defined by
$$
S(x):=|\xi_1|,\qquad x=\{\xi_1,\xi_2,\xi_3,\ldots \}\in l_1.
$$
Evidently, $S$ is majorized by $N_1$ on $l_1$. Further, we point out that for any $x =\{\xi_1,\xi_2,\xi_3,\ldots \}\in l_1$, the first entry of the sequence $F(x)$ is $\xi_1$, so $S$ is majorized by $N_2$ as well. It thus follows that $S$ is continuous at
zero, hence ubiquitously in $l_1$ with respect to $N_1$ and $N_2$.

As it stands now, we have displayed on $l_1$ a seminorm $S$ and two complete norms, $N_1$ and $N_2$, such that $S$ is ubiquitously continuous with respect to both norms. Hence, in order to attain our goal, it suffices to prove that $N_1$ and $N_2$ are non-equivalent.

Indeed, suppose to the contrary that $N_1$ and $N_2$ are equivalent, so that for some positive constants $\beta\leq\gamma$,
$$
\beta N_1(x) \leq N_2(x) \leq\gamma N_1(x),\qquad x\in l_1;
$$
that is,
$$
\beta N_1(x) \leq N_{\infty}(F(x)) \leq\gamma N_1(x),\qquad x\in l_1,
$$
or in other words, since $F(x - x') = F(x) - F(x')$ for all $x$ and $x'$ in $l_1$,
\begin{equation}\label{e5}
\beta N_1(x-x')\leq N_{\infty}(F(x)-F(x'))\leq\gamma N_1(x-x'),\qquad x,x'\in l_1.
\end{equation}

As we shall see, however, \eqref{e5} will imply that $F$ is an isomorphism between the Banach spaces $l_1$ and $c_0$.
And this, in turn, will lead to a contradiction since it is known (e.g., \cite[Corollary 2.1.6]{AK}) that $l_1$ and $c_0$ are
non-isomorphic.

We recall that a mapping from one Banach space to another is an {\it isomorphism} if it is a linear continuous bijection with a continuous inverse. Hence to prove that $F$ is an isomorphism between $l_1$ and $c_0$, it remains to show that $F$ and its inverse $F^{-1}$ are continuous.

To this end, let us first consider an arbitrary sequence $\{x_n\}_{n=1}^{\infty}$ in $l_1$ such that $x_n \to x'$ for some
$x'\in l_1$. Referring to the right inequality in \eqref{e5}, we get
$$
N_{\infty}(F(x_n)-F(x'))\leq\gamma N_1(x_n-x')\to 0\quad \mbox{as}\ n\to\infty.
$$
Hence $F(x_n)\to F(x')$, and it follows that $F$ is ubiquitously continuous in $l_1$.

Conversely, let $\{y_n\}_{n=1}^{\infty}$  be an arbitrary sequence in $c_0$ such that $y_n \to y'$ for some $y'\in c_0$. Since $F$ is a bijection from $l_1$ onto $c_0$, we can find a sequence $\{x_n\}_{n=1}^{\infty}$ and an element $x'$ in $l_1$ so that $y_n = F(x_n)$ and $y'=F(x')$. Whence, the left inequality in \eqref{e5} yields,
$$
\beta N_1(x_n-x')\leq N_{\infty}(F(x_n)-F(x'))\to 0\quad \mbox{as}\ n\to\infty.
$$
Thus $x_n \to x'$ or, otherwise put, $F^{-1}(y_n)\to F^{-1}(y')$; so $F^{-1}$ is ubiquitously continuous in $c_0$, and Example
\ref{ex4} is in our grasp.
}
\end{example}

As our paper draws to its end, we consider for the last time a nontrivial seminorm $S$ on an infinite-dimensional vector space $X$. Surely, the class of all norms on $X$ is the union of two distinct classes, $\mathcal{C}$ and $\mathcal{D}$, where $\mathcal{C}$ consists of all norms with respect to which $S$ is continuous, and $\mathcal{D}$ is the set of all norms with
respect to which $S$ is discontinuous. While the characterization of these two classes seems to be of interest, we have no clue on how to approach this job. Hence we leave it in the good hands of the reader as an open problem.


\end{document}